\documentclass[11pt]{amsart}

\usepackage{color}
\usepackage{xifthen}
\usepackage{amssymb}
\usepackage[table]{xcolor}
\usepackage{ytableau}
\usepackage{tikz}
\usepackage{varwidth}
\usepackage[all,cmtip]{xy}
\usetikzlibrary{calc, cd}
\usepackage[enableskew]{youngtab}
\usepackage{geometry}
\usepackage{enumerate}
\usepackage{array}
\usepackage[normalem]{ulem}
\usepackage{scalerel}
\usepackage{tikz-cd}
\usepackage[pagebackref]{hyperref}

\calclayout

\newtheorem{thm}{Theorem}[section]
\newtheorem{lemma}[thm]{Lemma}
\newtheorem{prop}[thm]{Proposition}

\newtheorem{cor}[thm]{Corollary}

\newtheorem{innercustomthm}{Theorem}
\newenvironment{customthm}[1]
  {\renewcommand\theinnercustomthm{#1}\innercustomthm}
  {\endinnercustomthm}

\theoremstyle{definition} 
\newtheorem{eg}[thm]{Example}

\newtheorem{defn}[thm]{Definition}

\newtheorem{rem}[thm]{Remark}
\newtheorem{qu}[thm]{Question}

\newcommand{\ds}{\displaystyle}

\newcommand{\NN}{\mathbb{N}}
\newcommand{\PP}{\mathbb{P}}

\newcommand{\ZZ}{\mathbb{Z}}

\newcommand{\cL}{\mathcal{L}}

\newcommand{\cO}{\mathcal{O}}

\newcommand{\on}{\operatorname}

\newcommand{\inv}{\on{inv}}

\DeclareMathOperator{\Pic}{Pic}

\newcommand{\iou}[1][]{
    \ifthenelse{\equal{#1}{}}{{\color{blue}\{IOU\}}}
    {{\color{blue}\{IOU: #1\}}}
}

\AtBeginDocument{%
   \def\MR#1{}
}

\title{ (Hurwitz--)Brill--Noether general marked graphs via the Demazure product }
\author[N. Pflueger]{Nathan Pflueger}\address{Department of Mathematics and Statistics, Amherst College}\email{npflueger@amherst.edu}
\date{\today}

\usepackage{mathdots}
\usetikzlibrary{arrows.meta,automata, decorations.pathreplacing, patterns}
\usepackage{bm}
\usepackage{multicol}

\newcommand{\sa}{s_\alpha}

\newcommand{\st}{s_\tau}
\newcommand{\sti}{s_{\tau^{-1}}}
\newcommand{\sbe}{s_\beta}
\newcommand{\sab}{s_{\alpha \star \beta}}
\newcommand{\svwd}{s^{v,w}_D}
\newcommand{\tvwd}{\tau^{v,w}_D}

\newcommand{\siom}{s_{\iota_m}}

\newcommand{\Inv}{\operatorname{Inv}}
\newcommand{\rg}{r_\Gamma}
\newcommand{\rga}{r_{\Gamma_1}}
\newcommand{\rgb}{r_{\Gamma_2}}

\newcommand{\sk}[1]{\sigma^k_{#1}}
\newcommand{\skm}{\sk{m}}
\newcommand{\ssk}[1]{s_{\sk{#1}}}
\newcommand{\sskm}{\ssk{m}}

\newcommand{\ts}[1]{\widetilde{\Sigma}_{#1}}

\newcommand{\bmu}{{\bm{\mu}}}

\theoremstyle{theorem}

\begin{document}
\maketitle

\begin{abstract}
This paper gives a novel and compact proof that a metric graph consisting of a chain of loops of torsion order $0$ is Brill--Noether general (a theorem of Cools--Draisma--Payne--Robeva), and a finite or metric graph consisting of a chain of loops of torsion order $k$ is Hurwitz--Brill--Noether general in the sense of splitting loci (a theorem of Cook-Powell--Jensen). In fact, we prove a generalization to (metric) graphs with two marked points, that behaves well under vertex gluing. The key construction is a way to associate permutations to divisors on twice-marked graphs, simultaneously encoding the ranks of every twist of the divisor by the marked points. Vertex gluing corresponds to the Demazure product, which can be formulated via tropical matrix multiplication.
\end{abstract}

%%%%%
%%%%%
\section{Introduction}
\label{sec:intro}

Combinatorial Brill--Noether theory, as inaugurated in \cite[\S 3]{bakerSpec},\cite{cdpr} concerns the census of all pairs $(\deg D, r(D))$ of degrees and ranks of divisors on a metric graph or finite graph $\Gamma$, often with the aim of studying algebraic curves that specialize to $\Gamma$. Hereafter, ``graph'' refers agnostically to either a metric or finite graph, and ``points'' of a graph are vertices for a finite graph, or any point along an edge for a metric graph. A \emph{divisor} on a graph is a $\ZZ$-linear combination of points, which is commonly viewed as a configuration of chips, some of which represent debt. An \emph{effective divisor} is a chip configuration with no debt; the \emph{degree} is the total number of chips; divisors are sorted into \emph{linear equivalence classes} via chip-firing moves, and the \emph{(Baker--Norine) rank} measures the degrees of freedom among the effective divisors in a given class. These definitions mirror analogous definitions for divisors on algebraic curves. The geometric and combinatorial versions of this census have a robust interface via specialization from curves to graphs, and lifting from graphs to curves, resulting in a rich intermingling of geometry, algebra, and combinatorics. We focus herein on the combinatorial story; geometric analogs of this content will appear in forthcoming work
\footnote{Although this forthcoming work is not finished, I am happy to provide a draft to interested readers.}
.

From the beginning, \emph{chains of loops} are celebrities in this story; these method actors perfectly assume their role as ``general algebraic curves,'' when the loops have general torsion order (Definition \ref{defn:torsionOrder}). More recently, they have starred in the new topic of Hurwitz--Brill--Noether theory, where torsion orders are taken to be the integer $k \geq 2$, and the chain gives an uncanny performance a ``general degree-$k$ cover of $\PP^1$.'' A very brief overview of Hurwitz--Brill--Noether theory, and references, are given in Section \ref{sec:splittingLoci}.

Newcomers often ask: \emph{why these graphs, specifically?} This paper offers an answer to that question. In a word: it is not chains of loops that are well-adapted to (Hurwitz--)Brill--Noether theory, but \emph{vertex gluings}. Define the vertex gluing of two twice-marked graphs $(\Gamma_1, v_1, w_2), (\Gamma_2, v_2, w_2)$ to be $(\Gamma, v_1, w_2)$, where $\Gamma$ is obtained by gluing $w_1$ to $v_2$; see Figure \ref{fig:vertexGluing}. The chain of loops is simply the end of the road: a graphs that is decomposed as much as possible via vertex gluing.
Our approach neatly folds together the Hurwitz--Brill--Noether story and Brill--Noether story.
We formulate in Definition \ref{defn:kgt} a genericity condition for twice-marked graphs called \emph{$k$-general transmission}. This paper's main results are as follows.

\begin{customthm}{A}
\label{thm:main}
If two twice-marked graphs have $k$-general transmission, then so does their vertex gluing.
A genus $1$ graph has $k$-general transmission if and only if it has torsion order $k$.
In particular, a chain of $k$-torsion loops has $k$-general transmission.
\hfill (Proof on page \pageref{proof:main})
\end{customthm}

\begin{customthm}{B}
\label{thm:bnGeneral}
If $(\Gamma, v, w)$ has $0$-general transmission, then it is Brill--Noether general in the following sense: for every degree-$d$ divisor $D$ on $\Gamma$, if $r(D) = r$ and $D$ has vanishing sequences $(a_i)_{i=0}^r$ and $(b_i)_{i=0}^r$ at $v$ and $w$ respectively, then the \emph{adjusted Brill--Noether number} is nonnegative:
$$\ds \rho = g - (r+1)(g-d+r) - \sum_{i=0}^r (a_i-i) - \sum_{i=0}^r (b_i - i) \geq 0.$$
Here, the \emph{vanishing sequence} of a divisor $D$ at a point $u$ is the sequence $(a_i)_{0 \leq i \leq r(D)}$ defined by $a_i = \max \{a \in \ZZ: r(D-au) \geq r-i\}$.
\hfill (Proof on page \pageref{proof:bnGeneral})
\end{customthm}

\begin{customthm}{C}
\label{thm:hbnGeneral}
If $k \geq 2$ and $(\Gamma, v, w)$ has $k$-general transmission, then $kv \sim kw$, $r(kv) \geq 1$, and $(\Gamma, v, w)$ is Hurwitz--Brill--Noether general, in the following sense: every divisor $D$ on $\Gamma$ belongs to a splitting type locus $W^\bmu(\Gamma)$ with respect to $F = kv$ such that $|\bmu| \leq g$. 
See Section \ref{sec:splittingLoci} for terminology on splitting loci.
\hfill (Proof on page \pageref{proof:hbnGeneral})
\end{customthm}

\begin{rem}
Theorem \ref{thm:bnGeneral} shows that $0$-general transmission implies Brill--Noether generality, but the converse is false; $0$-general transmission is a strictly stronger condition. 
Indeed, a chain of loops is Brill--Noether general provided that the torsion orders are either $0$ \emph{or sufficiently large}; see \cite{cdpr,pflChains} for precise statements.
\end{rem}

A weaker form of Theorem \ref{thm:bnGeneral}, incorporating vanishing orders at only one marked point, is proved in \cite{pflChains}.
By discarding all the vanishing order terms except $(a_r-r)$, Theorem \ref{thm:bnGeneral} implies the Brill--Noether generality condition proved for chains of loops in \cite[Theorem 1.1]{cdpr}; in particular, it implies that the Brill--Noether rank $w^r_d(\Gamma)$ in the sense of \cite{limPaynePotashnik} is at most $\rho$. 
The Hurwitz--Brill--Noether generality condition in Theorem \ref{thm:hbnGeneral} was proved for chains of loops by Cook-Powell and Jensen \cite[Theorem 1.2]{cpjMethods}; a weaker form, without the splitting loci terminology, was proved in \cite[Corollary 3.6]{pflKGonal}

The basis of our analysis is a way to associate permutations, which we call \emph{transmission permutations} to divisors on twice-marked graphs. These permutations contain all the information needed to compute ranks of divisors on chains of such twice-marked graphs.
A crucial advantage of the transmission permutation point of view is that it behaves well under vertex gluing, and thereby provides a route to constructing new (Hurwitz--)Brill--Noether general graphs of any genus: if a few more lower-genus graphs with $k$-general transmission are identified, then these can be mixed and matched like beads on a string to build personal, Brill--Noether general gifts for your special someone. One of my principal hopes in writing this paper is that it will lead to the identification of a broader class of (Hurwitz--)Brill--Noether general graphs.

%%%%%
%%%%%
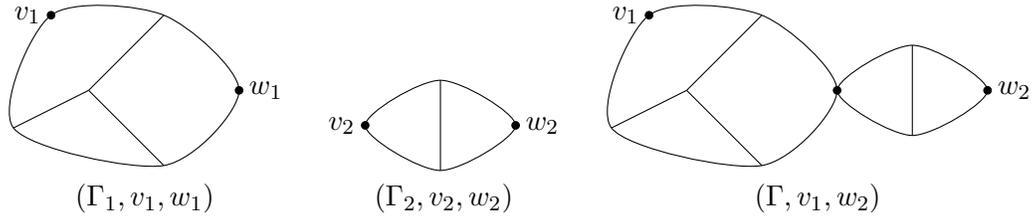
\begin{figure}
\begin{tabular}{ccc}

\begin{tikzpicture}
\coordinate (v1) at (-2.5,1);
\coordinate (w1) at (0,0);
\coordinate (a) at (-1,1);
\coordinate (l) at (-3,-0.5);
\coordinate (r) at (-1,-1);
\coordinate (c) at (-2,0);

\draw[fill] (v1) circle[radius=0.05] node[left] {$v_1$};
\draw[fill] (w1) circle[radius=0.05] node[right] {$w_1$};
\draw plot [smooth cycle]  coordinates {(v1) (a) (w1) (r) (l)};
\draw (a) -- (c);
\draw (l) -- (c);
\draw (r) -- (c);
\end{tikzpicture}
&
\begin{tikzpicture}
\coordinate (v2) at (w1);
\coordinate (w2) at (2,0);
\coordinate (a2) at (1,0.6);
\coordinate (b2) at (1,-0.6);
\draw plot [smooth cycle] coordinates {(v2) (a2) (w2) (b2)};
\draw (a2) -- (b2);
\draw[fill] (v2) circle[radius=0.05] node[left] {$v_2$};
\draw[fill] (w2) circle[radius=0.05] node[right] {$w_2$};

\end{tikzpicture}
&
\begin{tikzpicture}
\coordinate (v1) at (-2.5,1);
\coordinate (w1) at (0,0);
\coordinate (a) at (-1,1);
\coordinate (l) at (-3,-0.5);
\coordinate (r) at (-1,-1);
\coordinate (c) at (-2,0);

\coordinate (v2) at (w1);
\coordinate (w2) at (2,0);
\coordinate (a2) at (1,0.6);
\coordinate (b2) at (1,-0.6);

\draw[fill] (v1) circle[radius=0.05] node[left] {$v_1$};
\draw[fill] (w1) circle[radius=0.05] node[above] {};
\draw plot [smooth cycle]  coordinates {(v1) (a) (w1) (r) (l)};
\draw (a) -- (c);
\draw (l) -- (c);
\draw (r) -- (c);

\draw plot [smooth cycle] coordinates {(v2) (a2) (w2) (b2)};
\draw (a2) -- (b2);
\draw[fill] (w2) circle[radius=0.05] node[right] {$w_2$};

\end{tikzpicture}\\
$(\Gamma_1, v_1, w_1)$
&
$(\Gamma_2, v_2, w_2)$
&
$(\Gamma, v_1, w_2)$
\end{tabular}\\
\caption{
An example of a vertex gluing.
}
\label{fig:vertexGluing}
\end{figure}
%%%%%
%%%%%

\subsection{How do you quantify how special a divisor is?}

Brill--Noether theory of curves or graphs concerns the categorization of divisor classes according to ``how special'' they are, and predicts the codimension, within the Picard group, of the locus of equally special divisors. The word ``codimension'' can be taken literally within the Picard group of a curve or a metric graph, but must be taken figuratively for a finite graph.

 The subject began by categorizing divisor classes by degree and rank, or equivalently (using Riemann-Roch) by the two numbers $r(D)$ and $r(K-D)$, where $K$ denotes the canonical divisor. The expected codimension is $(r(D)+1)(r(K-D)+1)$; if $r = r(D)$, $d = \deg D$, and $g$ is the genus, this is equal to $(r+1)(g-d+r)$, an expression near and dear to acolytes of Brill--Noether theory. Since the advent of limit linear series in the 1980s, it is common to refine this classification by further sorting divisors according to their vanishing sequences at one or more marked points $p_1, \cdots, p_n$, which amounts to recording not just $r(D)$ but all values $r(D - n p_i)$ for $n \geq 0$. 
More recently, the subject of \emph{Hurwitz--Brill--Noether theory} aims to study the categorization of divisors on curves that have a degree-$k$ map $C \to \PP^1$, or alternatively curves with a chosen degree-$k$ divisor $F$ of degree $k$ and $r(F) \geq 1$. In this context, one classifies divisors according not just to $r(D)$, but $r(D-nF)$ for all $n \in \ZZ$ (both positive and negative). Both variations develop a theme: divisors are classified according to the ranks of some collection of ``twists.'' 

The present paper develops the theme further: if $u,v$ are two marked points, then we may categorize divisors according to the ranks $r(D + au - bv)$, $a,b \in \ZZ$, of all twists at the marked points. This can be used to study Hurwitz--Brill--Noether theory if the divisor $F$ is a combination of $u$ and $v$. This is a function of two variables, but for certain divisors, which we will call \emph{submodular} (see Section \ref{sec:submodular}), it is conveniently encoded in a permutation. The curious sign choice, writing $+au-bv$, is made so that more generic divisors have fewer inversions in their transmission permutations, and to simplify the later discussion of the Demazure product. In particular, the ``most generic'' divisors of degree $g$ give the identity permutation, and these divisors do not change the transmission permutation of another divisor on another twice-marked graph when attached by vertex gluing.

\begin{defn}
\label{defn:tp}
The \emph{transmission permutation} of $D$ on $(\Gamma,v,w)$, if it exists, is the unique permutation $\tvwd: \ZZ \to \ZZ$ such that the following two equivalent equations hold for all $a,b \in \ZZ$.
\begin{eqnarray}
r(D + a v - bw) + 1 &=& \# \{n \geq b: \tvwd(n) \leq a \} \label{eq:rd}\\
r(K_\Gamma - D -a v + bw) +1&=& \# \{ n < b: \tvwd(n) > a\} \label{eq:rkd}
\end{eqnarray}
\end{defn}

\begin{eg}
The ``most generic'' transmission permutations are those with no inversions at all, i.e. increasing permutations. Indeed, suppose that $D$ is a degree $d$ divisor on a genus $g$ graph such that all ``twists'' $D' = D + av - bw$ are nonspecial (that is, either $|D'| = \emptyset$ or $|K_\Gamma-D'| = \emptyset$). Then $\tau^{v,w}_D(n) = n + g - d$ for all $n \in \ZZ$; we will refer to this increasing permutation as $\iota_{d-g}$. In particular, the identity permutation corresponds to the ``most generic possible'' divisors with $d=g$.
\end{eg}

\begin{eg}
Consider the simple transposition of $m$ and $m+1$; we denote this later by $\sigma^0_m$. This permutation has one inversion, and $\tau^{v,w}_D = \sigma^0_m$ if and only if $\deg D = g$, $r(D + mv - (m+1)w) = 0$, and all other twists $D' = D + av - bw$ are nonspecial. If $\Gamma$ is a cycle and $v-w$ is nontorsion, then there is a unique such divisor class $D$, given by a single point $w_m \in \Gamma$. These points form an equally spaced sequence, with $w_{-1} = v$ and $w_0 = w$ (cf. Figure \ref{fig:genus1}). If $v-w$ is torsion, we obtain instead an affine permutation $\sigma^k_m$, as defined below. See Section \ref{ssec:genus1} for a complete analysis for transmission permutations in genus $1$, from which the claims in this example follow.
\end{eg}

\begin{rem}
\label{rem:tauEquiv}
As stated above, Equations \eqref{eq:rd} and \eqref{eq:rkd} are equivalent. This can be seen as follows. Since $r(D + av - bw) \to -1$ as $b \to \infty$, Equation \eqref{eq:rd} is equivalent to the difference equation
$$r(D+av-bw) - r(D+av-(b+1)w) = \delta( \tvwd(b) \leq a) \mbox{ for all } a,b \in \ZZ.$$
By Riemann-Roch, this is equivalent to a dual difference equation
$$r(K_\Gamma - D-av+(b+1)w) - r(K_\Gamma-D-av+bw) = \delta( \tvwd(b) > a) \mbox{ for all } a,b \in \ZZ,$$
which similarly is equivalent to Equation \eqref{eq:rkd}. Also, these equations imply that the \emph{function} $\tvwd$ is necessarily a \emph{permutation}, since for any $b \ll 0$, Riemann-Roch implies
$$\# \{ n \geq b: \tvwd(n) = a\} = r(D+av-bw) - r(D+(a-1)v-bw) = 1.$$
\end{rem}

The transmission permutation has many pleasant properties. First, the ``expected codimensions'' from all variations of Brill--Noether theory described above are equal to the cardinality of some subset of the \emph{inversions} of $\tau^{u,v}_D$. Second, transmission permutations are extremely well-suited to vertex gluing. There is an associative product $\star$, called the \emph{Demazure product}, such that when a divisor $D$ on a vertex gluing is split across the two graphs as $D_1 + D_2$, then we have an identity (Theorem \ref{thm:chaining})
$$
\tau^{v_1,w_2}_D = \tau^{v_1, w_2}_{D_1} \star \tau^{v_2, w_2}_{D_2},
$$
provided that all divisors involved are submodular. This identity is the basis for all the analysis in this paper, and the specific definition of $\tau^{v,w}_D$ made above (including the sign choice) is engineered to keep it as naturally stated as possible.

%%%%%
%%%%%
\subsection{The groups $\ts{k}$}
\label{ssec:tsk}
We are interested in twice-marked graphs in which transmission permutations belong to a specific family of groups. Example \ref{eg:chainLoops} will show that in fact the group $\ts{k}$ is precisely the set of transmission permutations that occur on chains of $k$-torsion cycles.

\begin{defn}
\label{defn:tsk}
Let $\ts0$ denote the group of permutations $\alpha: \ZZ \to \ZZ$ with finitely many inversions. For $k \geq 2$, let $\ts{k}$ denote the group of permutations $\alpha: \ZZ \to \ZZ$ such that $\alpha(n+k) = \alpha(n) + k$ for all $n \in \ZZ$. The groups $\ts{k}$ for $k\geq2$ are called \emph{extended affine symmetric groups
\footnote{The (un-extended) \emph{affine symmetric group} is the subgroup $\widetilde{S}_k$ of $\ts{k}$ in which $\sum_{n=0}^{k-1} (\alpha(n) -n ) = 0$. This is equivalent to $\{n \geq 0: \alpha(n) < 0 \} = \{n < 0: \alpha(n) \geq 0\}$. Equations \eqref{eq:rd} and \eqref{eq:rkd} 
imply that, if such an $\alpha$ is $\tau^{v,w}_D$, then $r(D-v) - r(K_\Gamma-D+v) = 0$, and thus $\deg D - g = 0$ by Riemann-Roch. So the (unextended) affine permutations are those that can arise via divisors with degree equal to the genus.
}
.}
\end{defn}

If $(\Gamma,v,w)$ is twice-marked graph with $kv \sim kw$, then one can show that $\tvwd(n+k) = \tvwd(n) + k$ for all submodular divisors $D$ and $n \in \ZZ$; this is why the groups $\ts{k}$ arise naturally in our context.

\begin{defn}
\label{defn:invk}
An \emph{inversion} of a permutation $\alpha$ is a pair $(u,v) \in \ZZ^2$ such that $u<v$ and  $\alpha(u) > \alpha(v)$. The set of inversions of $\alpha$ is denoted $\Inv(\alpha)$.
If $\alpha \in \ts{k}$, call two inversions $(u,v),(u',v')$ $k$-equivalent if there is an integer $n$ such that $(u',v') = (u + nk, v+ nk)$. Denote by $\inv_k(\alpha)$ the number of $k$-equivalence classes of $\Inv(\alpha)$. More concretely, $\inv_0(\alpha)$ is the number of inversions, and for $k \geq 2$, $\inv_k(\alpha)$ is the number of inversions $(u,v)$ with $0 \leq v < k$.
\end{defn}

The number $\inv_k$ measures word length for a convenient system of generators.

\begin{defn}
Denote by $\iota_m$ the shift permutation $\iota_m(n) = n-m$. Let $\sigma^k_m \in \ts{k}$ denote the permutation exchanging $n$ and $n+1$ for all $n \equiv m \pmod{k}$, and fixing all other integers (when $k=0$, congruence means equality).
%When $k=0$, $\sigma^0_m$ is the simple transposition of $m$ and $m+1$.
We call the permutations $\sigma^k_m$ \emph{simple reflections.}
\end{defn}

Note that $\inv_k (\iota_m) = 0$, and $\inv_k (\sigma^k_m) = 1$. In fact, for all $\alpha \in \ts{k}$,
\begin{equation}
\label{eq:invas}
\inv_k( \alpha \sigma^k_m) = \inv_k(\alpha) + \begin{cases}
1 & \mbox{ if } \alpha(m) < \alpha(m+1)\\
-1 & \mbox{ if } \alpha(m) > \alpha(m+1).
\end{cases}
\end{equation}
Equation \eqref{eq:invas} shows that $\ts{k}$ is generated by $\iota_1$ and $\{\sigma^k_m: m \in \ZZ\}$; this could have been our definition of $\ts{k}$, and it shows more clearly why $\ts0$ belongs to the same family as $\ts{k}$ for $k \geq 2$.

\subsection{$k$-general transmission}

If $\tau$ is a transmission permutation on $(\Gamma, v, w)$ the number $\rho = g - \inv_k(\tau)$ plays the role of the ``Brill--Noether number,''and $\inv_k(\tau)$ is the ``expected codimension.'' This is reflected in the key definition of this paper:

\begin{defn}
\label{defn:kgt}
Let $k \geq 0$, with $k \neq 1$.
A genus-$g$ twice-marked graph $(\Gamma, v, w)$ has \emph{$k$-general transmission} if every divisor $D$ on $\Gamma$ is submodular, and satisfies $\tau^{v,w}_D \in \ts{k}$ and  $\inv_k \tau^{v,w}_D \leq g$.
\end{defn}

%%%%%
%%%%%
\subsection{Conventions}
\label{ssec:conventions}
Throughout this paper, the word \emph{graph} refers \emph{either} to a connected metric graph or a finite connected graph with no loop edges, unless stated otherwise . A \emph{twice-marked graph} $(\Gamma, v, w)$ is a graph with two chosen points (vertices, if $\Gamma$ is a finite graph). We denote the canonical divisor by $K_\Gamma$, linear equivalence by $D \sim E$, and the Baker-Norine rank by $r(D)$.
When a graph $\Gamma$ is clear from context, $g$ will denote its genus.

The symbol $\NN$ denotes the set of \emph{nonnegative} numbers. The symbol $\delta$ will always be used for an indicator function; e.g. $\delta(n \geq 5)$ is $1$ if $n \geq 5$ and $0$ otherwise.

%%%%%%%%%%
%%%%%%%%%%
\section{Submodular divisors}
\label{sec:submodular}

This section gives a convenient description of the divisors which possess transmission permutations as well as a handy formula for computing the permutations, and apply it to completely describe the situation in genus $1$.

\begin{defn}
\label{defn:submodular}
For any divisor on $(\Gamma,v,w)$,denote for convenience
\begin{equation*}
\Delta(D) = r(D) - r(D - v) - r(D-w) + r(D-v-w)
\end{equation*}
The divisor $D$ is \emph{submodular}\footnote{
The reader might object that this should be called ``supermodular'' given the direction of the inequality. The reason is that the function $(a,b) \mapsto r(D + av - bw)$, and \emph{that} function is submodular. See also Remark \ref{rem:whySubmod}.
}
with respect to $v,w$ if $\Delta(D') \geq 0$ for all twists $D' = D + av-bw$.
\end{defn}

\begin{rem}
\label{rem:basepoints}
Since $r(D') - r(D'-w)\in \{0,1\}$, submodularity is equivalent to the implication
$$\mbox{for all twists } D' = D + av -bw,\hspace{0.5cm} r(D'-w) = r(D') \hspace{0.25cm} \Rightarrow \hspace{0.25cm} r(D'-v-w) = r(D'-v).$$
This is always true for a divisor on an \emph{algebraic curve} $C$, since if every divisor in the complete linear series $|D'|$ contains $w$, then so does every divisor in $|D'-v|$ (assuming $v \neq w$).
\end{rem}

\begin{prop}
\label{prop:tauFormula}
The transmission permutation $\tvwd$ exists if and only if $D$ is submodular with respect to $v,w$. If $D$ is submodular, then $\tvwd$ is given by the formula
$$\tvwd(b) = \min \{a \in \ZZ: r(D+av - bw) > r(D + av - (b+1)w) \}.$$
Equivalently, $\tvwd(b)$ is the unique $a \in \ZZ$ for which $\Delta(D + av-bw) = 1$.
\end{prop}

\begin{proof}
Suppose that $\tvwd$ exists. Equation \eqref{eq:rd} and an inclusion-exclusion argument shows that $\Delta(D + av - bw) = \delta(\tvwd(b) = a) \geq 0$ for all $a,b \in \ZZ$, hence $D$ is submodular.

Conversely, suppose that $D$ is submodular with respect to $v,w$, and define $\tau : \ZZ \to \ZZ$ by 
$\tau(b) = \min \{a \in \ZZ: r(D+av - bw) > r(D + av - (b+1)w) \}.$
%$\tau(b) = \min \{a \in \ZZ: \svwd(a+1,b) > \svwd(a,b) \}.$
Remark \ref{rem:basepoints} implies that
\begin{equation}
\label{eq:rdDiff}
r(D+av - bw) = r(D + av - (b+1)w) + \delta( \tau(b) \leq a )
%\svwd(a,b) = \svwd(a,b+1) + \delta( \tau(b) < a )
\end{equation}
which implies $\tau = \tvwd$ by Remark \ref{rem:tauEquiv}.
\end{proof}

%%%%%
%%%%%
\subsection{The base case: genus $1$}
\label{ssec:genus1}

Let $(\Gamma, v, w)$ be a twice-marked graph of genus $1$, with $v \not\sim w$.

\begin{defn}
\label{defn:torsionOrder}
The torsion order of $(\Gamma,v,w)$ is the nonnegative integer $k$ such that $nv \sim nw$ if and only if $n \in k \ZZ$. By our assumption that $v \not\sim w$, $k \neq 1$.
\end{defn}

More concrelely, if $\Gamma$ is a loop (also known as a cycle), and the two edges from $v$ to $w$ have lengths $\ell_1, \ell_2$, then $k$ is the minimum postive integer such that $k \ell_1 \in (\ell_1 + \ell_2) \ZZ$, or $0$ if no such $m$ exists. If $\Gamma$ is not a loop, then it contains a loop as as a subgraph (the \emph{skeleton}), and one can define $\ell_1, \ell_2$ by contracting $v,w$ to the skeleton.

\begin{lemma}
\label{lem:genus1}
Suppose $(\Gamma, v, w)$ has torsion order $k$, and let $D$ be a degree $d$ divisor on $\Gamma$. Then $D$ is submodular, and
\begin{enumerate}
\item If there exists $m \in \ZZ$ such that $D \sim m w + (d-m) v$, then $\tau^{v,w}_D = \iota_{d-1} \sigma_{m-1}^{k}$.
\item If no such $m$ exists, then $\tau^{v,w}_D = \iota_{d-1}$.
\end{enumerate}
\end{lemma}

See Figure \ref{fig:genus1} for an illustration of this Lemma, in the case $k=7$. A divisor of degree $1$ is linearly equivalent to a unique single point on the skeleton. There is an equally spaced sequence of points giving the transmission permutations $\sigma^k_m$, and any point not in this sequence has transmission permutation $\iota_0$. Divisors of degree other than $1$ are equivalent to a multiple of $v$ plus a single point on the skeleton, and the transmission permutation is as describe before, but shifted by $d-1$.

\newcommand{\sevencycle}{
\coordinate (v) at (-1,1.5);
\coordinate (w) at (1,1.5);
\coordinate (v06) at ({sin(4*360/7)},{-cos(4*360/7)});
\coordinate (v00) at ({sin(3*360/7)},{-cos(3*360/7)});
\coordinate (v01) at ({sin(2*360/7)},{-cos(2*360/7)});
\coordinate (v02) at ({sin(1*360/7)},{-cos(1*360/7)});
\coordinate (v03) at ({sin(0*360/7)},{-cos(0*360/7)});
\coordinate (v04) at ({sin(-1*360/7)},{-cos(-1*360/7)});
\coordinate (v05) at ({sin(-2*360/7)},{-cos(-2*360/7)});
\draw (v06) -- (v00) -- (v01) -- (v02) -- (v03) -- (v04) -- (v05) -- cycle;
\draw (v06) -- (v) node[left] {$v$};
\draw (v00) -- (w) node[right] {$w$};
}
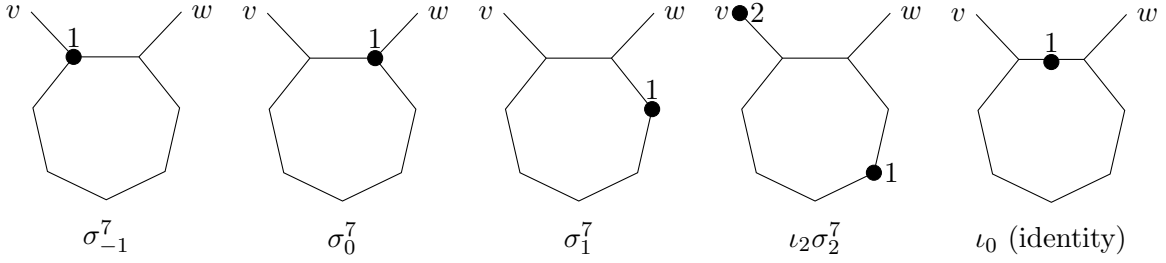
\begin{figure}
\begin{tikzpicture}
\sevencycle
\draw[fill] (v06) circle[radius=0.1] node[above] {$1$};
\draw (0,-1.5) node {$\sigma^7_{-1}$};
\end{tikzpicture}
\begin{tikzpicture}
\sevencycle
\draw[fill] (v00) circle[radius=0.1] node[above] {$1$};
\draw (0,-1.5) node {$\sigma^7_{0}$};
\end{tikzpicture}
\begin{tikzpicture}
\sevencycle
\draw[fill] (v01) circle[radius=0.1] node[above] {$1$};
\draw (0,-1.5) node {$\sigma^7_{1}$};
\end{tikzpicture}
\begin{tikzpicture}
\sevencycle
\draw[fill] (v) circle[radius=0.1] node[right] {$2$};
\draw[fill] (v02) circle[radius=0.1] node[right] {$1$};
\draw (0,-1.5) node {$\iota_2 \sigma^7_{2} $};
\end{tikzpicture}
\begin{tikzpicture}
\sevencycle
\draw[fill] (0,{-cos(3.5*360/6)}) circle[radius=0.1] node[above] {$1$};
\draw (0,-1.5) node {$\iota_0$ (identity)};
\end{tikzpicture}
\caption{Examples of transmission permutations on a genus $1$ twice-marked graph, with torsion order $7$.}
\label{fig:genus1}
\end{figure}

\begin{proof}
For any $D' = D + av - bw$, Riemann-Roch implies $r(D') = -1+\max \{0, \deg D' \} + \delta(D' \sim 0)$, and therefore, in the notation of Definition \ref{defn:submodular}, we deduce the following equations for $\Delta(D')$.
%$$
%\Delta(D') = \begin{cases}
%\delta(D' \sim 0) & \mbox{ if } \deg D' \leq 0,\\
%\delta(D' \not\sim v,w) & \mbox{ if } \deg D' = 1,\\
%\delta(D' \sim v+w) & \mbox{ if } \deg D' \geq 2.
%\end{cases}
%$$

\begin{align*}
\mbox{ if $\deg D' \leq 0$, then } \Delta(D') %&= r(D') - r(D'-v) - r(D'-w) + r(D'-v-w)\\
&= (-1+\delta(D' \sim 0)) - (-1) - (-1) + (-1)\\
&= \delta(D' \sim 0)\\
\mbox{ if $\deg D'  = 1$, then } \Delta(D') &= 0 - (-1 + \delta(D' -v \sim 0)) + (-1+\delta(D'-w \sim 0)) + (-1)\\
&= 1 - \delta(D' \sim v) - \delta(D' \sim w) \\
&= \delta(D' \not\sim v,w)\\
\mbox{ if $\deg D' \geq 2$, then } \Delta(D') &= 1 - 0 - 0 + (-1 + \delta(D'-v-w \sim 0))\\
&= \delta(D' \sim v + w)
\end{align*}

So $D$ is submodular, and by Proposition \ref{prop:tauFormula}
$$
\tvwd(b) = \begin{cases}
b-d & \mbox{ if } D + (b-d) v - bw \sim 0,\\
b-d+2 & \mbox{ if } D + (b-d+2)v - bw \sim v + w,\\
b-d+1 & \mbox{ if } D + (b-d+1)v - bw \not\sim v,w.
\end{cases}
$$
Define $\Lambda = \{m \in \ZZ: D \sim m w + (d-m)v \}$. This is either empty, or an arithmetic progression $m + k \ZZ$ for some $m$. The formula above may be conveniently rewritten
$$\tvwd(b) = b-d+1 - \delta(b \in \Lambda) + \delta(b+1 \in \Lambda).$$
This is precisely $\iota_{d-1} \sigma^k_{m-1} (b)$ if $\Lambda = m + k \ZZ$, and $\iota_{d-1}$ otherwise, which establishes the lemma.
\end{proof}

Lemma \ref{lem:genus1} shows that $(\Gamma,v,w)$ has $k$-general transmission. In fact, it also shows that $(\Gamma,v,w)$ does \emph{not} have $k'$-general transmission for any $k' \neq k$, because for example $\tau^{v,w}_w = \sigma^k_0$; if $k \nmid k'$ then $\sigma^k_0 \not\in \ts{k'}$, while if $k \mid k'$ then $\inv_{k'} \sigma^k_0 = k'/k \geq 2$.

\begin{cor}
\label{cor:genus1GT}
$(\Gamma,v,w)$ has $k$-general transmission if and only if it has torsion order $k$.
\end{cor}

%%%%%
%%%%%
\section{Demazure products}
\label{sec:demazureGeneral}

The heart of this paper is the observation that transmission permutations (when they exist) behave well under vertex gluing: we will prove in Theorem \ref{thm:chaining} that they compose according to a certain associative operation which we refer to as the \emph{Demazure product.} This operation plays the role in this paper that displacement of partitions plays in \cite{pflChains, cpjComponents, cpjMethods}.

\begin{defn}
For any permutation $\tau: \ZZ \to \ZZ$ and integers $a,b$ define
$$\st(a,b) = \# \{ n \geq b:\ n < a\}.$$
Call $\tau$ \emph{almost-sign-preserving} if $\st(0,0)$ and  $\sti(0,0)$ are finite. This is equivalent to saying that both $\st(a,b), \sti(a,b)$ are finite for \emph{all} $a,b \in \ZZ$, i.e. $\st$ is a function $\ZZ^2 \to \NN$.
\end{defn}

\begin{rem}
If a divisor $D$ on $(\Gamma, v,w)$ is submodular with $\tau = \tvwd$, then in this notation, 
$$r(D + av-bw) +1 = s_{\tau}(a+1,b) \mbox{ and } r(K_\Gamma-D-av+bw) +1 = s_{ \tau^{-1}}(b,a+1).$$
The ``$+1$''s are due to the strict inequality $n<a$ in the definition of $\st$.
\end{rem}

\begin{defn}
For any two functions $s_1, s_2: \ZZ^2 \to \NN$, define a function $s_1 \star s_2$ by
$$s_1 \star s_2(a,b) = \min_{\ell \in \ZZ} s_1(a,\ell) + s_2(\ell,b).$$
For two almost-sign-preserving permutations $\alpha, \beta$, the \emph{Demazure product} $\alpha \star \beta$, if it exists, is the unique almost-sign-preserving permutation satisfying
$\sab = \sa \star \sbe.$
\end{defn}

Every Coxeter group possesses a product called either the \emph{$0$-Hecke product } or \emph{Demazure product}, which is obtained by setting $q=0$ in the Hecke algebra, as defined for example in \cite[\S 6]{bjornerBrenti}. We will prove below that, when restricted to a symmetric group or affine symmetric group (which are Coxeter groups), our definition of $\star$ recovers this usual Demazure product.

This formula for $s_1 \star s_2$ may be viewed as tropical matrix multiplication of two infinite matrices $s_1, s_2$, with rows and columns indexed by $\ZZ$. 
The fact that the usual Demazure product on the symmetric group can be characterized by tropical matrix multiplication is discussed in \cite{cpRR} in slightly different notation and without full details, but as far as I can tell it was not previously known. A forthcoming paper will prove that in fact $\alpha \star \beta$ exists for \emph{all} almost-sign-preserving $\alpha,\beta$. To simplify the present paper, we will only prove its existence on $\ts{k}$; this is done in Section \ref{ssec:demazure}.

\begin{lemma}
\label{lem:whereMin}
Suppose $\alpha, \beta$ are almost-sign-preserving and $a,b \in \ZZ$. The minimum value $\sa \star \sbe (a,b)$ of $ \{ \sa(a,\ell) + \sbe(\ell,b): \ell \in \ZZ\}$ is obtained for some $\ell$ such that
$\beta^{-1}(\ell-1) < b \leq \beta^{-1}(\ell)$.
\end{lemma}

\begin{proof}
Denote $L = \{ \ell \in \ZZ: \sa \star \sbe(a,b) = \sa(a,\ell) + \sbe(\ell,b) \}$. This set is bounded above; choose $\ell_0 \in L$ such that $\ell_0+1 \not\in L$. Since $\sa(a,\ell_0+1) \leq \sa(a,\ell_0)$, we have $\sbe(\ell_0+1,b) > \sbe(\ell_0,b)$, i.e. $\beta^{-1}(\ell_0) \geq b$. Now, let $\ell$ be the minimum integer such that all of $\beta^{-1}(\ell), \beta^{-1}(\ell+1), \cdots, \beta^{-1}(\ell_0)$ are at least $b$. Then $\sbe(\ell,b) < \cdots < \sbe(\ell_0,b)$, or equivalently $s(\ell,b) = s(\ell_0,b) - (\ell_0-\ell)$. 
Since $\ell(a,\ell) \leq \ell(a,\ell_0) + (\ell_0-\ell)$, we have $\sa(a,\ell) + \sbe(\ell,b) \leq \sa(a,\ell_0) + \sbe(\ell_0,b)$ and therefore $\ell \in L$ as well. By construction, $\beta^{-1}(\ell-1) < b \leq \beta^{-1}(\ell)$.
\end{proof}

\begin{lemma}
\label{lem:starIota}
For all almost-sign-preserving $\alpha$ and $m \in \ZZ$, $\alpha \star \iota_m$ exists and is $\alpha \iota_m$.
\end{lemma}

\begin{proof}
Given $a,b \in \ZZ$, the only $\ell$ for which $\iota_m^{-1}(\ell-1) < b \leq \iota^{-1}_m(\ell)$ is $\ell = b-m$. Therefore Lemma \ref{lem:whereMin} implies
$\sa \star \siom(a,b) = \sa(a,b-m) + \siom(b-m,b) = \sa(a,b-m) = s_{\alpha \iota_m}(a,b)$.
\end{proof}

%%%%%
%%%%%
\subsection{Demazure products in $\ts{k}$}
\label{ssec:demazure}

Fixed $k \geq 0$ with $k \neq 1$. We prove in this subsection

\begin{thm}
\label{thm:gamma}
If $\alpha, \beta \in \ts{k}$, then the Demazure product $\alpha \star \beta$ exists, lies in $\ts{k}$, and satisfies $\inv_k \alpha \star \beta \leq \inv_k \alpha + \inv_k \beta$. It is characterized on generators by the equations
\begin{equation}
\label{eq:starSigma}
\alpha \star \iota_m = \alpha \iota_m
\hspace{0.5cm} \mbox{and} \hspace{0.5cm}
\alpha \star \sigma^k_m = \begin{cases}
\alpha \sigma^k_m & \mbox{ if } \alpha(m) < \alpha(m+1),\\
\alpha & \mbox{ if } \alpha(m) > \alpha(m+1).
\end{cases}
\end{equation}
\end{thm}

The second part of Equation \eqref{eq:starSigma}, along with associativity, can be taken as the \emph{definition} of the standard Demazure product on (affine) symmetric groups, and (with appropriate modifications) Coxeter groups more generally; see e.g. \cite[\S 6.1]{bjornerBrenti}. So Theorem \ref{thm:gamma} confirms that our definition of $\star$, via tropical matrix multiplication, extends the usual one on (affine) symmetric groups.

\begin{lemma}
\label{lem:sasskm}
If $\alpha$ is almost-sign-preserving, then for any simple reflection $\sigma^k_m$,
\begin{equation}
\label{eq:sasskm}
\sa \star \sskm (a,b) = \begin{cases}
\sa(a,b) & \mbox{ if } b \not\equiv m+1 \pmod{k},\\
\min \{ \sa(a,b-1), \sa(a,b+1) +1 \} & \mbox{ if } b \equiv m+1 \pmod{k}.
\end{cases}
\end{equation}
\end{lemma}

\begin{proof}
First suppose $b \not\equiv m+1 \pmod{k}$. Then $\skm(\ell) \geq b$ if and only if $\ell \geq b$, so $\skm(\ell-1) < b \leq \skm(\ell)$ if and only if $\ell = b$. Since $\sskm(b,b) = 0$,
Lemma \ref{lem:whereMin} implies that $\sa \star \sskm(a,b) = \sa(a,b)$.

Now suppose $b \equiv m+1 \pmod{k}$. Then $\skm(\ell) \geq b$ if and only if $\ell = b-1$ or $\ell \geq b+1$. Therefore $\skm(\ell-1) < b \leq \skm(\ell)$ if and only if $\ell \in \{b-1, b+1\}$. By definition of $\sskm$, $\sskm(b-1,b) = 0$ and $\sskm(b+1,b) = 1$. So Lemma \ref{lem:whereMin} implies the desired equation.
\end{proof}

\begin{lemma}
\label{lem:starSigma}
For $\alpha \in \ts{k}$ and any $\sigma^k_m$,
$
\sa \star \sskm = \begin{cases}
s_{\alpha \skm} & \mbox{ if } \alpha(m) < \alpha(m+1),\\
\sa & \mbox{ if } \alpha(m) > \alpha(m+1).
\end{cases}
$
\end{lemma}

\begin{proof}
First, observe that $s_{\alpha \skm}(a,b) = \# \{n: \skm(n) \geq b \mbox{ and } \alpha(n) < a\},$ and $\{n: \skm(n) \geq b\}$ is either $\{n: n \geq b\}$ or $\{n: n \geq b\} \backslash \{b\} \cup \{b-1\}$, depending on whether $b \equiv m+1 \pmod{k}$. Hence
\begin{align}\begin{split}
\label{eq:saskm}
s_{\alpha \sigma^k_m}(a,b) = \sa(a,b) &+ \delta(b \equiv m+1 \pmod{k} \mbox{ and } \alpha(b-1) < a \leq \alpha(b))\\
&- \delta( b \equiv m+1 \pmod{k} \mbox{ and } \alpha(b) < a \leq \alpha(b-1).
\end{split}\end{align}

Observe that
$\sa(a,b-1) = \sa(a,b) + \delta(\alpha(b-1) < a)$ and $\sa(a,b+1) +1 = \sa(a,b) + \delta(\alpha(b) \geq a)$. This implies
$ \min \{ \sa(a,b-1), \sa(a,b+1) + 1 \} = \sa(a,b) + \delta\Big( \alpha(b-1) < a \leq \alpha(b)\Big),$
and therefore
\begin{equation}
\label{eq:sasskm2}
\sa \star \sskm (a,b) = \sa(a,b) + \delta\Big(  b \equiv m+1 \pmod{k} \mbox{ and } \alpha(b-1) < a \leq \alpha(b) \Big).
\end{equation}
If $\alpha(m) > \alpha(m+1)$, then $\alpha \in \ts{k}$ implies that $\alpha(b-1) > \alpha(b)$ for all $b \equiv m+1 \pmod{k}$, so Equation \eqref{eq:sasskm2} shows that $\sa \star \sskm = \sa$ in that case. On the other hand, if $\alpha(m) < \alpha(m+1)$ then $\alpha(b-1) < \alpha(b)$ for all such $b$, and Equation \eqref{eq:saskm} implies that $\sa \star \sskm = s_{\alpha \skm}$ in this case.
\end{proof}

\begin{proof}[Proof of Theorem \ref{thm:gamma}]
The existence of $\alpha \star \iota_m$ and $\alpha \star \sigma^k_m$, and the formulas in Equation \eqref{eq:starSigma}, follow from Lemmas \ref{lem:starIota} and \ref{lem:starSigma}. It remains to show that for all $\alpha, \beta \in \ts{k}$, $\alpha \star \beta$ exists and $\inv_k \alpha \star \beta \leq \inv_k \alpha + \inv_k \beta$. We prove this by induction on $\inv_k \beta$. If $\inv_k \beta = 0$, then $\beta = \iota_m$ for some $m \in \ZZ$, and we have already considered this case. If $\inv_k \beta > 0$, then by Equation \eqref{eq:invas} there exists a simple reflection $\sigma^k_m$ such that $\beta' = \beta \sigma^k_m$ satisfies $\inv_k \beta' = \inv_k \beta -1$ and $s_\beta = s_{\beta'} \star \sskm$. By inductive hypothesis, $\alpha \star \beta'$ exists, so $\sa \star \sbe = \sa \star s_{\beta'} \star \sskm = s_{\alpha \star \beta'} \star \sskm$. By Lemma \ref{lem:starSigma}, this is either $s_{(\alpha \star \beta') \skm}$ or $s_{\alpha \star \beta'}$; either way $\alpha \star \beta$ exists, and $\inv_k(\alpha \star \beta) \leq \inv_k(\alpha \star \beta') + 1 \leq \inv_k(\alpha) + \inv_k(\beta)$.
\end{proof}

%%%%%
%%%%%
\subsection{Vertex gluing}
\label{ssec:chainability}

We now relate the Demazure product to transmission permutations, and prove Theorem \ref{thm:main}.
The following notation will be useful.

\begin{defn}
The \emph{transmission function of $D$ with respect to $v,w$} is $s^{v,w}_D: \ZZ^2 \to \NN$ defined by 
$$s^{v,w}_D(a,b) = r(D + (a-1)v - bw)+1.$$
\end{defn}

\begin{rem}
\label{rem:whySubmod}
The divisor $D$ is submodular if and only if $\svwd(a,b) - \svwd(a+1,b) - \svwd(a,b+1) + \svwd(a+1,b+1) \leq 0$ for all $a,b \in \ZZ$, which is why we use that word. If $\svwd$ is submodular, then 
$$\svwd = s_{\tvwd}.$$
The seemingly unnecessary ``$-1$'' in the definition of $\svwd$ is included to keep this statement clean.
\end{rem}

Fix two twice-marked graphs $(\Gamma_1, v_1, w_1)$ and $(\Gamma_2, v_2, w_2)$, and let $(\Gamma, v_1, w_2)$ be the vertex gluing. In what follows, we will be working with divisors on $\Gamma_1, \Gamma_2,$ and $\Gamma$, so we will use subscripts to clarify the graph in question, e.g. we write $r_{\Gamma_1}(D_1)$ to indicate the rank of $D_1$ when it is viewed as a divisor on $D_1$.
To declutter the prose, we follow a convention: when we write ``$D = D_1 + D_2$,'' we mean implicitly that $D_i$ is supported on $\Gamma_i$. 
We will deduce Theorem \ref{thm:main} from the following.

\begin{thm}
\label{thm:chaining}
For any divisor $D = D_1 + D_2$ on the vertex gluing $(\Gamma,v_1,w_2)$ described above,
$$s^{v_1, w_2}_D = s^{v_1, w_1}_{D_1} \star s^{v_2, w_2}_{D_2}.$$
 In particluar, if $D_1, D_2$ and $D$ are submodular, then 
$$\tau^{v_1,w_2}_D = \tau^{v_1,w_1}_{D_1} \star \tau^{v_2, w_2}_{D_2}.$$
In both equations, terms involving $D_i$ ($i \in 1,2$) are understood to refer to it \emph{as a divisor on $\Gamma_i$}.
\end{thm}

\begin{lemma}
\label{lem:glueSim}
Let $D = D_1 + D_2$ and $E = E_1 + E_2$ be two divisors on $\Gamma$, with $\deg D_i = \deg E_i$ for $i=1,2$. Then $D \sim_\Gamma E$ if and only if both $D_1 \sim_{\Gamma_1} E_1$ and $D_2 \sim_{\Gamma_2} E_2$.
\end{lemma}

\begin{proof}
This follows from the observations that a divisor on $\Gamma$ is principal if and only if it is the sum of a principal divisor on $\Gamma_1$ and a principal divisor on $\Gamma_2$.
%, which in turn can be deduced by comparing the divisor of a rational function $f: \Gamma \to \RR$ to those of its restrictions to $\Gamma_1, \Gamma_2$.
\end{proof}

\begin{lemma}
\label{lemma:effectiveGlue}
A divisor $D = D_1+D_2$ on $\Gamma$ has $\rg(D) \geq 0$ if and only if there exists $\ell \in \ZZ$ such that both $\rga(D_1-\ell w_1) \geq 0$ and $\rgb(D_2 + \ell v_2) \geq 0$.
\end{lemma}

\begin{proof}
If $E = E_1 + E_2$ is effective of the same degree as $D$, and $\ell = \deg D_1 - \deg E_1$, then Lemma \ref{lem:glueSim} shows that $D \sim_\Gamma E$  if and only if both $D_1  - \ell w_1 \sim_{\Gamma_1} E_1$ and $D_2 + \ell v_2 \sim_{\Gamma_2} E_2$.
\end{proof}

\begin{lemma}
\label{lem:glueHelper}
Let $D = D_1 +D_2$ be a divisor on $\Gamma$. For any integer $r$, $r_\Gamma(D) \geq r$ if and only if for all integers $u \in \{0,1,\cdots, r\}$, there exists $\ell \in \ZZ$ such that $r_{\Gamma_1}(D_1 -\ell w_1) \geq u$ and $r_{\Gamma_2}(D_2 + \ell v_2) \geq r-u$.
\end{lemma}

\begin{proof}
The inequality $\rg(D) \geq r$ means that for all divisors $E = E_1 + E_2$, where $E_1, E_2$ are both effective, $\rg((D_1 - E_1) + (D_2 - E_2)) \geq 0$. Let $f(E_1)$ denote the maximum $\ell \in \ZZ$ such that $\rga(D_1 - \ell w_1 - E_1) \geq 0$, and let $g(E_2)$ denote the minimum $\ell \in \ZZ$ such that $\rgb(D_2 + \ell v_2 - E_2) \geq 0$. Lemma \ref{lem:glueHelper} shows that $\rg(D - E) \geq 0$ if and only if $f(E_1) \geq g(E_2)$. 

For $u \in \{0, \cdots, r\}$, let $S(u)$ be the statement: for all $E = E_1 + E_2$, with $E_1, E_2$ effective of degrees $u, r-u$, $\rg(D-E) \geq r$. 
So $r_\Gamma(D) \geq r$ if and only if all of $S(0), \cdots, S(r)$ are true.
Then $S(u)$ is true if and only if $f(E_1) \geq g(E_2)$ for all degree $u$ $E_1$ and degree $r-u$ $E_2$. Since $E_1, E_2$ can be chosen independently, $S(u)$ is true if and only if there exists $\ell \in \ZZ$ such that $f(E_1) \geq \ell \geq g(E_2)$ for all such $E_1, E_2$, which is equivalent to the pair of inequalities $\rga(D_1 - \ell w_1) \geq u, \rgb(D_2  + \ell v_2) \geq r-u$.
\end{proof}

The following result provides the basic tool for studying divisors on vertex gluings, which may be understood as saying that ranks of such divisors can be computed as a tropical dot product. 
Special cases of this result have previously been used in \cite[Lemma 3.14]{pflChains} and \cite[Proposition 5.1]{borziWeierstrass}.

\begin{prop}
\label{prop:glueRank}
For a divisor $D = D_1 + D_2$ on $\Gamma$,
$$
\rg(D) = \min_{\ell \in \ZZ} \rga(D_1 - (\ell+1) w_1) + \rgb(D_2 + \ell v_2) + 1.
$$
\end{prop}

\begin{proof}
First note that Lemma \ref{lem:glueHelper} remains true if we replace ``$u \in \{0, 1, \cdots, r\}$'' with ``$ u \geq 0$'' since for any $u \geq r+1$, the inequality $\rgb(D_2 + \ell v_2) \geq r-u$ is true for \emph{all} $\ell \in \ZZ$. For all $u \geq 0$, define
$$f(u) = \max \{\ell \in \ZZ: \rga(D_1 - \ell w_1) = u \}.$$
There exists \emph{some} $\ell$ such that $\rga(D_1 - \ell w_1) \geq u$ and $\rgb(D_2 + \ell v_2) \geq r-u$ if and only if $\ell = f(u)$ is such an integer. Therefore we may rewrite
$$\rg(D) = \min_{u \geq 0}\ u + \rgb(D_2 + f(u) v_2).$$
Turning to the right side of the claimed equation, consider the function 
$$s(\ell) = \rga(D_1 - (\ell+1)w_1) + \rgb(D_2 + \ell v_2) + 1.$$
For all $u \geq 0$, the maximality of $f(u)$ implies $\rga(D_1 - (f(u)+1)w_1) = u-1$, so $s(f(u)) = u + \rgb( D_2 + f(u) v_2)$. On the other hand, if $\ell$ is an integer that is \emph{not} equal to $f(u)$ for any $u \geq 0$, then $\rga(D_1 - (\ell+1)w_1) = \rga(D_1 - \ell w_1)$, and therefore $s(\ell-1) \leq s(\ell)$. It follows that $\min_{\ell \in \ZZ} s(\ell) = \min_{u \geq 0} s(f(u))$. We have seen that the right side is equal to $\rg(D)$.
\end{proof}

\begin{proof}[Proof of Theorem \ref{thm:chaining}] For all $a,b \in \ZZ$, we may split $D + (a-1) v_1 - b w_2$ across $\Gamma_1$ and $\Gamma_2$ as $(D_1 + (a-1)v_1) + (D_2 - b w_2)$. Proposition \ref{prop:glueRank} implies
\begin{eqnarray*}
s^{v_1, w_2}_D(a,b) &=& \min_{\ell \in \ZZ} r_{\Gamma_1} (D_1 + (a-1) v_1 - (\ell+1) w_1) + r(D_2 +\ell v_2 - b w_2),
\end{eqnarray*}
and this is $ \ds \min_{\ell \in \ZZ} s^{v_1, w_1}_{D_1} (a,\ell+1) + s^{v_2, w_2}_{D_2}(\ell+1,b) = s^{v_1, w_1}_{D_1} \star s^{v_2, w_2}_{D_2}(a,b)$.
\end{proof}

\begin{proof}[Proof of Theorem \ref{thm:main}] \label{proof:main}
Suppose that $(\Gamma_1, v_1, w_1)$ and $(\Gamma_2, v_2, w_2)$ both have general transmission. Let $g_i$ be the genus of $\Gamma_i$; the genus of $\Gamma$ is then $g_1 + g_2$. By Theorems \ref{thm:gamma} and \ref{thm:chaining}  every divisor $D = D_1 + D_2$ is submodular, its transmission permutation lies in $\ts{k}$, and $\inv_k \tau^{v_1, w_2}_{D} \leq \inv_k \tau^{v_1, w_1}_{D_1} + \inv_k \tau^{v_2, w_2}_{D_2} \leq g_1 + g_2$, so $(\Gamma, v_1, w_2)$ has $k$-general transmission. Corollary \ref{cor:genus1GT} verified that $k$-general transmission is equivalent to $k$-torsion in genus $1$; by induction a chain of $k$-torsion loops has $k$-general transmission.
\end{proof}

\begin{eg}
\label{eg:chainLoops}
Consider the chain $(\Gamma, v_1, w_g)$ of $k$-torsion loops considered in \cite{pflKGonal, jensenRanganathan, cpjComponents, cpjMethods}, consisting of a chain of marked cycles $(E_i,v_i,w_i)$, $i=1,\cdots,g$, in which there is a path of length $1$ and a path of length $k-1$ between $v_i$ and $w_i$ in cycle $E_i$, and $w_i$ is glued to $v_{i+1}$ for $i=1,\cdots,g-1$. For now, regard $\Gamma$ as a \emph{finite} graph. Denote by $\langle \xi \rangle_i \in E_i$ the point located $i$ units clockwise from $w_i$ (when the length $1$ path to $v_i$ is drawn on top); note that $\langle \xi \rangle_i \sim w_i + \xi (w_i-v_i)$ as a degree-$1$ divisor on $E_i$. Every degree $g$ divisor on $\Gamma$ is linearly equivalent to a unique \emph{break divisor} of the form $D = \sum_{i=1}^g \langle \xi_i \rangle_i$, consisting of a single chip on each cycle. Combining Lemma \ref{lem:genus1} and Theorem \ref{thm:chaining} gives the following formula.
\begin{equation}
\label{eq:chainBreak}
\tau^{v_1,w_g}_D = \sigma^k_{\xi_1} \star \sigma^k_{\xi_2} \star \cdots \star \sigma^k_{\xi_g}
\end{equation}
In particular, if $\tau$ is in the affine symmetric group, then the set of divisors with $\tau^{v_1,w_g}_D = \tau$ is in bijection with reduced words for $\tau$ in the generators $\sigma^k_1, \cdots, \sigma^k_k$. In particular, every element $\tau \in \ts{k}$ occurs as a transmission permutation on such a chain.

If we instead regard $\Gamma$ as a \emph{metric} graph, we must allow $\xi_i$ to take non-integer values, but the only change needed to Equation \eqref{eq:chainBreak} is that $\sigma^k_{\xi_i}$ should be regarded as the identity permutation when $\xi_i$ is not an integer. This analysis generalizes in a straightforward manner to any chain of loops, including those with varying torsion orders. One must show that the Demazure products are well-defined in this case, which will be done in a forthcoming paper.
\end{eg}

%%%%%
%%%%%
\section{$0$-general transmission and Brill--Noether generality}
\label{sec:bnGeneral}

We now focus on $k=0$ and demonstrate that $0$-general transmission implies a form of Brill--Noether generality.
Theorem \ref{thm:bnGeneral} follows immediately from the definition of $0$-general transmission, and the following. \label{proof:bnGeneral}

\begin{lemma}
\label{lem:invBound}
Let $D$ be a divisor on $\Gamma$ that is submodular with respect to $v,w$, let $r = r(D)$, and abbreviate $\tau^{v,w}_D$ by $\tau$.
Let $(a_i)_{0 \leq i \leq r}, (b_i)_{0 \leq i \leq r}$ be the vanishing sequences at $v,w$ respectively.
Then
$$\inv_0 \tau \geq (r+1)(g-d+r) + \sum_{i=0}^r (a_i-i) + \sum_{i=0}^r (b_i-i).$$
\end{lemma}

%%%%%
%%%%%
\begin{figure}
\begin{tabular}{ccc}
\begin{tikzpicture}[scale=0.2]
%shaded regions
\fill[lightgray] (-0.5, 0.5) rectangle (-7,8);

\draw[->] (-7,0) -- (8,0) node[right] {$n$};
\draw[->] (0,-6) -- (0,8) node[above] {$\tau(n)$};

% The main quandrant par of alpha
\node [fill=black, circle, inner sep=1pt] at (0,-3) {};
\node [fill=black, circle, inner sep=1pt] at (2,-1) {};
\node [fill=black, circle, inner sep=1pt] at (5,-2) {};
\draw (5,-2) node[below right] {$(b_i, -a_{\sigma(i)})$};
\draw[dotted] (4.5,-1.5) rectangle (5.5,-2.5);
\node [fill=black, circle, inner sep=1pt] at (6,0) {};

% The  opposite quadrant
\node [fill=black, circle, inner sep=1pt] at (-2,1) {};
\node [fill=black, circle, inner sep=1pt] at (-1,2) {};

% upper-right
\node [fill=black, circle, inner sep=1pt] at (1,3) {};
\node [fill=black, circle, inner sep=1pt] at (3,4) {};
\node [fill=black, circle, inner sep=1pt] at (4,5) {};
\node [fill=black, circle, inner sep=1pt] at (7,6) {};
%\node at (8.5,8) {$\iddots$};

% lower-left
\node [fill=black, circle, inner sep=1pt] at (-3,-4) {};
%\node at (-5,-5) {$\iddots$};

% labels of regions sizes
\end{tikzpicture}
&
\begin{tikzpicture}[scale=0.2]
%shaded regions
\fill[lightgray] (-0.5, -1.5) rectangle (-7,0.5);

\draw[->] (-7,0) -- (8,0) node[right] {$n$};
\draw[->] (0,-6) -- (0,8) node[above] {$\tau(n)$};

% The main quandrant par of alpha
\node [fill=black, circle, inner sep=1pt] at (0,-3) {};
\node [fill=black, circle, inner sep=1pt] at (2,-1) {};
\node [fill=black, circle, inner sep=1pt] at (5,-2) {};
\draw (5,-2) node[below right] {$(b_i, -a_{\sigma(i)})$};
\draw[dotted] (4.5,-1.5) rectangle (5.5,-2.5);
\node [fill=black, circle, inner sep=1pt] at (6,0) {};

% The  opposite quadrant
\node [fill=black, circle, inner sep=1pt] at (-2,1) {};
\node [fill=black, circle, inner sep=1pt] at (-1,2) {};

% upper-right
\node [fill=black, circle, inner sep=1pt] at (1,3) {};
\node [fill=black, circle, inner sep=1pt] at (3,4) {};
\node [fill=black, circle, inner sep=1pt] at (4,5) {};
\node [fill=black, circle, inner sep=1pt] at (7,6) {};
%\node at (8.5,8) {$\iddots$};

% lower-left
\node [fill=black, circle, inner sep=1pt] at (-3,-4) {};

\end{tikzpicture}

&
\begin{tikzpicture}[scale=0.2]
%shaded regions
\fill[lightgray] (-0.5, 0.5) rectangle (4.5,8);

\draw[->] (-7,0) -- (8,0) node[right] {$n$};
\draw[->] (0,-6) -- (0,8) node[above] {$\tau(n)$};

% The main quandrant par of alpha
\node [fill=black, circle, inner sep=1pt] at (0,-3) {};
\node [fill=black, circle, inner sep=1pt] at (2,-1) {};
\node [fill=black, circle, inner sep=1pt] at (5,-2) {};
\draw (5,-2) node[below right] {$(b_i, -a_{\sigma(i)})$};
\draw[dotted] (4.5,-1.5) rectangle (5.5,-2.5);
\node [fill=black, circle, inner sep=1pt] at (6,0) {};

% The  opposite quadrant
\node [fill=black, circle, inner sep=1pt] at (-2,1) {};
\node [fill=black, circle, inner sep=1pt] at (-1,2) {};

% upper-right
\node [fill=black, circle, inner sep=1pt] at (1,3) {};
\node [fill=black, circle, inner sep=1pt] at (3,4) {};
\node [fill=black, circle, inner sep=1pt] at (4,5) {};
\node [fill=black, circle, inner sep=1pt] at (7,6) {};
%\node at (8.5,8) {$\iddots$};

% lower-left
\node [fill=black, circle, inner sep=1pt] at (-3,-4) {};
%\node at (-5,-5) {$\iddots$};

\end{tikzpicture}
\\
$|S| = (g-d+r)(r+1)$
&
$|A_i| = a_{\sigma(i)}-\sigma(i)$
& 
$|B_i| = b_i - i$
\end{tabular}
\caption{The sets in Lemma \ref{lem:invBound}. In this example, $r=3$, $g-d+r = 2$, and $(a_0, a_1, a_2, a_3) = (0, 1,2,3)$ and $(b_0, b_1, b_2, b_3) = (0,2,5,6)$.
}
\label{fig:SSi}
\end{figure}
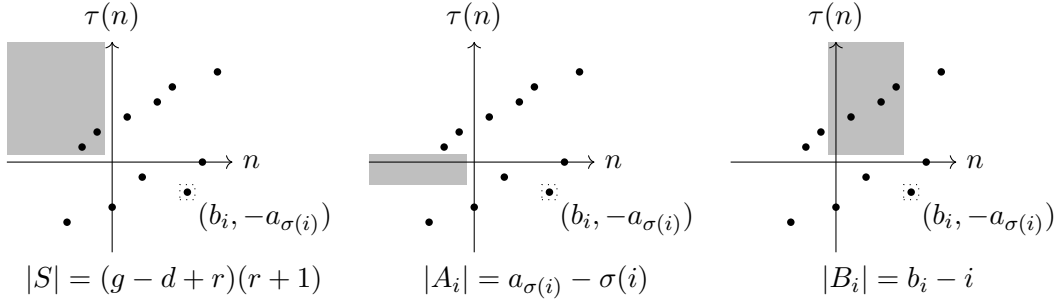

\begin{proof}
The sets of vanishing orders are $\{a_0, \cdots, a_r\} = \{a \geq 0: r(D - av) > r(D - (a+1)v)$ and $\{b_0, \cdots, b_r \} = \{b \geq 0: r(D - bw) > r(D-(b+1)w)\}$. By Equation \eqref{eq:rd}, these sets are $\{n \geq 0: \tau^{-1}(-n) \geq 0\}$ and $\{ n \geq 0: \tau(n) \leq 0 \}$, respectively. Let $T = \{ (n, \tau(n)):\ n \geq 0 \mbox{ and } \tau(n) \leq 0\}$. It follows that there is a permutation $\sigma$ of $\{0, \cdots, r\}$ such that $T = \{ (b_i, -a_{\sigma(i)}):\ 0 \leq i \leq r\}$.

Fix a choice $0 \leq i \leq r$. We will bound the number of integers $n$ such that $(n, b_i) \in \Inv_0(\tau)$. The set of such $n$ certainly includes each of the following three disjoint sets, illustrated in Figure \ref{fig:SSi}.
\begin{eqnarray*}
S &=& \{ n < 0:\ \tau(n) > 0 \}\\
A_i &=& \{ n < 0: 0 \geq \tau(n) >  \tau(b_i) \}\\
B_i &=& \{ n \geq 0:\  n < b_i,\ \tau(n) > 0 \}
\end{eqnarray*}
Equation \eqref{eq:rkd} implies that $|S| = g-d+r$, and our description of the sets of vanishing orders above implies that $|A_i| = a_{\sigma(i)}-\sigma(i)$ and $|B_i| = b_i - i$. Summing over $i$ gives
$
\inv_0(\tau) \geq \sum_{i=0}^{r} \left( |S| + |A_i| + |B_i| \right) 
= \sum_{i=0}^r (g-d+r) + \sum_{i=0}^r (a_{\sigma(i)} - \sigma(i) ) + \sum_{i=0}^r (b_i - i)
$
and therefore $\inv_0(\tau) \geq (r+1)(g-d+r) + \sum_{i=0}^r (a_i-i) + \sum_{i=0}^r (b_i-i)$, as desired.
\end{proof}

%%%%%
%%%%%
\section{$k$-general transmission and splitting loci}
\label{sec:splittingLoci}

This section briefly summarizes the notion of \emph{splitting type loci} on graphs, which play the role of $W^r_d$ in Hurwitz--Brill--Noether theory.
Initial work on Hurwitz--Brill--Noether theory \cite{cm99, cm02, pflKGonal, jensenRanganathan} concerned the same census as Brill--Noether theory, but for a general $k$-gonal curve: which degrees and ranks occur? Later, Cook-Powell--Jensen \cite{cpjComponents, cpjMethods} and Larson \cite{larsonHBN} independently realized that this census could be refined, and one should instead classify the \emph{splitting types} of divisors on $k$-gonal curves.
This means that line bundles $\cL$ on a curve with degree-$k$ cover $\pi: C \to \PP^1$ are classified by the isomorphism class of $\pi_\ast \cL$. This is a rank-$k$ vector bundle on $\PP^1$, so there exists a unique nondecreasing $k$-tuple $\bmu = (\mu_1, \cdots, \mu_k)$ such that $\pi_\ast \cL \cong \cO_C(\mu_1) \oplus \cdots \oplus \cO_C(\mu_k)$. This $k$-tuple $\bmu$, called the splitting type, determines the rank and degree of $\cL$ and considerably more information. Splitting types also have the virtue of determining irreducible loci.
For a comprehensive discussion of Hurwitz--Brill--Noether theory, see the recent \emph{tour de force} \cite{larsonLarsonVogt}, which also makes systematic use of the affine symmetric group and explains its role in detail.
On the combinatorial side, ``splitting'' cannot be taken so literally, but it is still possible to define splitting type loci numerically. 

\begin{defn}
Let $\bmu = (\mu_1, \cdots, \mu_k)$ be a nondecreasing $k$-tuple of integers. Define 
$$
x_m(\mu) = \sum_{i=1}^k \max \{0, \mu_i + m + 1\}
$$
for all $m \in \ZZ$, and let $d(\bmu) = g-1 + \sum_{i=1}^k (\mu_i + 1)$. 

Let $D$ be a divisor on a graph $\Gamma$, and $F$ a degree-$k$ divisor with $r(F) \geq 1$. We say that $D$ has  \emph{splitting type $\bmu$} (with respect to $F$) if its class belongs to the \emph{splitting type locus}
$$
W^\bmu(\Gamma) = \{ [D] \in \Pic^{d(\bmu)}(\Gamma): r(D + m F) = x_m(\bmu) -1\mbox{ for all } m \in \ZZ \}.
$$
The \emph{expected codimension} of this locus is $|\bmu| = \sum_{i<j} \max \{0, \mu_j - \mu_i - 1\}$. 
\end{defn}

%Note that this definition does depend on the choice of divisor class $F$, even though it is left implicit in the notation $W^\bmu(\Gamma)$. This is a vestige of the fact that, for a $k$-gonal algebraic curve, the gonality pencil is unique.
%Typically, one also considers the closed loci $\overline{W}^{\bmu}(\Gamma)$, where the equation $r(D + mF) \geq x_m(\bmu)$ is replaced by an inequality. For our purposes we focus on the strict loci $W^\bmu$.

It is not obvious that every divisor $D$ on a graph even \emph{has} a splitting type, and indeed this may not be true! 
The issue is analous the the issue that not all divisors on twice-marked graphs are submodular: a splitting type $\bmu$ such that $[D] \in W^\bmu(\Gamma)$ exists if and only if the difference $r(D + mF) - r(D+(m-1)F)$ is nondecreasing in $m$.
Part of Theorem \ref{thm:hbnGeneral} is that $k$-general transmission guarantees that all divisors $D$ meet this condition.

We now specialize to our application: suppose $(\Gamma, v, w)$ is a twice-marked graph with $k$-general transmission, and let $F = k v$. 

\begin{lemma}
\label{lem:Fg1k}
With the assumptions above, $kv \sim kw$ and the rank of $F = kv$ is at least $1$.
\end{lemma}
\begin{proof}
Note that $kv \sim kw$ is equivalent to $r(kv-kw) \geq 0$, so we must show that $r(kv-kw) \geq 0$ and $r(kv) \geq 1$.
Since $\Gamma$ has $k$-general transmission, the empty divisor $D=0$ has a transmission permutation $\tau \in \ts{k}$. Since $r(D) = 0$, it follows that there is a (unique) $m \in \ZZ$ such that $m \geq 0, \tau(m) \leq 0$. Since $\tau(m+k) = \tau(m) + k$, it follows that $m+k \in \{ n \geq k: \tau(n) \leq k \}$ and $\{m, m+k\} \subseteq \{ n \geq 0: \tau(n) \leq k \}$. Equation \eqref{eq:rd} implies that $r(kv - kw) \geq 0$ and $r(kv) \geq 1$.
\end{proof}

%%%%%
%%%%%
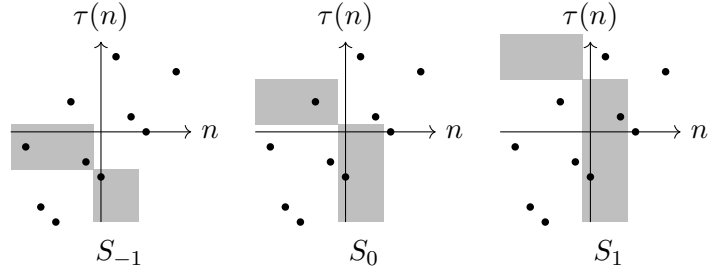
\begin{figure}
\begin{tabular}{ccc}
\begin{tikzpicture}[scale=0.2]
%shaded regions
\def\m{-1}
\fill[lightgray] (-0.5, \m*3+0.5) rectangle (2.5,-6);
\fill[lightgray] (-0.5, \m*3+0.5) rectangle (-6,\m*3+3.5);

\draw[->] (-6,0) -- (6,0) node[right] {$n$};
\draw[->] (0,-6) -- (0,6) node[above] {$\tau(n)$};

% The graph of the ext affine perm (-3,5,1), k=3
\foreach \m in {-1,...,1} {
\node [fill=black, circle, inner sep=1pt] at (0+3*\m,-3+3*\m) {};
}
\foreach \m in {-2,...,0} {
\node [fill=black, circle, inner sep=1pt] at (1+3*\m,5+3*\m) {};
}
\foreach \m in {-2,...,1} {
\node [fill=black, circle, inner sep=1pt] at (2+3*\m,1+3*\m) {};
}
\end{tikzpicture}
&
\begin{tikzpicture}[scale=0.2]
%shaded regions
\def\m{0}
\fill[lightgray] (-0.5, \m*3+0.5) rectangle (2.5,-6);
\fill[lightgray] (-0.5, \m*3+0.5) rectangle (-6,\m*3+3.5);

\draw[->] (-6,0) -- (6,0) node[right] {$n$};
\draw[->] (0,-6) -- (0,6) node[above] {$\tau(n)$};

% The graph of the ext affine perm (-3,5,1), k=3
\foreach \m in {-1,...,1} {
\node [fill=black, circle, inner sep=1pt] at (0+3*\m,-3+3*\m) {};
}
\foreach \m in {-2,...,0} {
\node [fill=black, circle, inner sep=1pt] at (1+3*\m,5+3*\m) {};
}
\foreach \m in {-2,...,1} {
\node [fill=black, circle, inner sep=1pt] at (2+3*\m,1+3*\m) {};
}
\end{tikzpicture}
&
\begin{tikzpicture}[scale=0.2]
%shaded regions
\def\m{1}
\fill[lightgray] (-0.5, \m*3+0.5) rectangle (2.5,-6);
\fill[lightgray] (-0.5, \m*3+0.5) rectangle (-6,\m*3+3.5);

\draw[->] (-6,0) -- (6,0) node[right] {$n$};
\draw[->] (0,-6) -- (0,6) node[above] {$\tau(n)$};

% The graph of the ext affine perm (-3,5,1), k=3
\foreach \m in {-1,...,1} {
\node [fill=black, circle, inner sep=1pt] at (0+3*\m,-3+3*\m) {};
}
\foreach \m in {-2,...,0} {
\node [fill=black, circle, inner sep=1pt] at (1+3*\m,5+3*\m) {};
}
\foreach \m in {-2,...,1} {
\node [fill=black, circle, inner sep=1pt] at (2+3*\m,1+3*\m) {};
}
\end{tikzpicture}
\\
$S_{-1}$
&
$S_0$
& 
$S_1$
\end{tabular}
\caption{The construction of $S_m$ in the proof of Theorem \ref{thm:hbnGeneral}. In this example, $k=3$, and $\tau(0) = -3,\ \tau(1) = 5,\ \tau(2) = 1$.
}
\label{fig:Sm}
\end{figure}
%%%%%
%%%%%

\begin{proof}[Proof of Theorem \ref{thm:hbnGeneral}] \label{proof:hbnGeneral}
Suppose $(\Gamma, v, w)$ has $k$-general transmission. The first two claims of the theorem are proved in Lemma \ref{lem:Fg1k}. It suffices to prove that for every divisor $D$ on $\Gamma$ with transmission permutation $\tau$, $D$ belongs to a splitting type locus $W^\bmu(\Gamma)$ such that $|\bmu| \leq \inv_k \tau$. Fix a divisor $D$ on $\Gamma$, and let $\tau = \tau^{v,w}_D$.
Define, for all $m \in \ZZ$,
$x_m = r(D + mF).$
Then $x_m - x_{m-1} = \# \{n: 0 \leq n < k, \tau(n) \leq mk \}$. The sets in this last expression are nested, so $x_m - x_{m-1}$ is nondecreasing in $m$. By the discussion above Lemma \ref{lem:Fg1k}, there exists a splitting type $\bmu$ such that
$r(D+mF) = x_m(\bmu)-1$ for all $m \in \ZZ$. This implies that $\deg D = d(\bmu)$, since for $m \gg 0$ we have $r(D + mF) = \deg D + mk -g$ and $x_m(\bmu) = \sum_{i=1}^k \mu_i + mk + k = d(\bmu) +mk - g$. Therefore
 $[D] \in W^\bmu(\Gamma)$. It remains to prove that $|\bmu| \leq \inv_k \tau$.

For every $m \in \ZZ$, define the following subset of $\Inv(\tau)$, which are illustrated in Figure \ref{fig:Sm}.

\newcommand{\floor}[1]{\left\lfloor #1 \right\rfloor}
$$S_m = \left\{ (i,j):\  i < 0 \leq j < k,\  \floor{\frac{\tau(j)-1}{k} } < m = \left\lfloor \frac{\tau(i)-1}{k} \right\rfloor \right\}$$

The definition immediately implies that the $S_m$ are pairwise disjoint. Their cardinalities are:
\begin{align*}
|S_m| %&=& \# \left\{ i < 0:\ \floor{ \frac{\tau(i)-1}{k} } = m \right\} \cdot \# \left\{ j:\ 0 \leq j < k, \floor{ \frac{\tau(j)-1}{k} } < m \right\}\\
&= \left\{ i < 0:\  mk < \tau(i) \leq (m+1)k \right\} \cdot \# \left\{ j:\ 0 \leq j < k, \tau(j) \leq mk \right\}\\
&= \left( k - \# \left\{i \geq 0:\ mk < \tau(i) \leq (m+1)k \right\} \right) \cdot \# \left\{ j:\ 0 \leq j < k, \tau(j) \leq mk \right\}\\
&= \left( k - r( D + (m+1)kv) + r(D + mkv) \right) \cdot \left( r(D+mkv) - r(D+mkv - kw) \right)\\
&= (k - x_{m+1} + x_m) \cdot (x_m - x_{m-1})
\end{align*}
By definition of $\bmu$, $x_m - x_{m-1} = \# \{i: \mu_i \geq -m\}$, and similarly $x_{m+1} - x_m = \# \{i: \mu_i \geq-m-1\}$, so $k - x_{m+1} + x_m = \# \{i: \mu_i < -m-1\}$. Therefore $|S_m| = \# \{ (i,j): \mu_i < -m-1 < \mu_j \}$. Summing these numbers for all $m$, the pair $(i,j)$ is counted precisely $\max \{ \mu_j - \mu_i - 1, 0 \}$ times, so $\sum_{m \in \ZZ} |S_m| = |\bmu|$, and $\inv_k \tau \geq |\bmu|$ as desired.
\end{proof}

%%%%%%%%%%
%%%%%%%%%%
\section{Questions for further work}
\label{sec:questions}

We conclude with a few questions naturally suggested by the two-pointed Brill--Noether theory of graphs studied in this paper.

\begin{qu}
Which permutations $\tau: \ZZ \to \ZZ$ occur as transmission permutation on graphs? The methods of this paper imply that all permutations with finitely many inversions and all extended affine permutations occur, but many other permutations are certainly possible, even for chains of loops (by choosing different torsion orders on different loops).
\end{qu}

\begin{qu}
The Brill--Noether existence question for \emph{finite} graphs \cite[Conjecture 3.10(1)]{bakerSpec} is still widely open; let's open it even wider. If $\tau \in \ts{0}$, does every twice-marked \emph{finite} graph $(\Gamma, v, w)$ of genus $g \geq \inv_0 \tau$ possess a divisor $D$ with $\svwd(a,b) \geq \st(a,b)$ for all $a,b \in \ZZ$?
\end{qu}

\begin{qu}
Are there natural criteria under which a submodular divisor on a graph lifts to a divisor with the \emph{same} transmission permutation on an algebraic curve? For the chain of loops with general torsion, lifting results have been obtain in \cite{cjp} for the case with no marked points, and \cite{heLifting} in the one-marked point case. For a chain with $k$-torsion, lifting results for certain splitting loci were obtained in \cite{jensenRanganathan}.
\end{qu}

\begin{qu}
Is it possible to classify (in some reasonable sense) the twice-marked graphs for which all divisors are submodular? Such twice-marked graphs better reflect the behavior of twice-marked algebraic curves.
\end{qu}

\begin{qu}
Which transmission functions $\svwd$ can occur for non-submodular divisors?
\end{qu}

%%%%%%%%%%
%%%%%%%%%%
\section*{Acknowledgements}
This work was supported by a Miner D. Crary Sabbatical Fellowship from Amherst College.
I am grateful to Sam Payne for conversations that encouraged me to write this paper, and Dave Jensen for helpful comments on an early draft.

%-----------
%-----------

\bibliographystyle{amsalpha}
\bibliography{../template/main}
\end{document}